\documentclass[12pt]{article}
\usepackage{amsmath}
\usepackage{amsfonts}
\usepackage{amssymb}
\usepackage{graphicx}%
\usepackage{mathrsfs}
\usepackage{algorithm}
\usepackage{algorithmic}
\usepackage{subfigure}
\usepackage{epsfig}
\usepackage{booktabs}
\usepackage{multirow}
\usepackage{rotating}
\usepackage{color}
\usepackage{CJK}
\usepackage{amssymb}
\usepackage{stmaryrd}
\usepackage{tipa}
\usepackage{amsfonts,amsmath,latexsym}
\usepackage{mathrsfs}  
\usepackage{amsbsy}
\usepackage{indentfirst}
\usepackage{amsmath}
\usepackage{color}
\usepackage{titlesec}

\usepackage{palatino,eulervm}
\def\exp {\textsf{exp}}

\def\a.s{\textsf{a. s}}

\def\and{\textsf{and}}

\newtheorem{theorem}{Theorem}

\newtheorem{lemma}{Lemma}

\newcommand{\beqnar}{\begin{eqnarray*}}
\newcommand{\eeqnar}{\end{eqnarray*}}
\newcommand{\ba}{\begin{array}}
\newcommand{\ea}{\end{array}}

\newcommand{\cu}{\boldsymbol}

\newcommand{\bx}{\preccurlyeq}
\newcommand{\lam}{\lambda_{\mathrm{max}}}
\newcommand{\qw}{\mathbb{E}}
\newcommand{\tre}{\mathrm{trexp}}

\newcommand{\dis}{\displaystyle}
\usepackage[numbers]{natbib}
\setcitestyle{open={},close={}}

\newenvironment{proof}[1]{\begin{trivlist}\item {\it
\bf Proof.}\quad} {\qed\end{trivlist}}

\newcommand{\qed}{\nopagebreak\hspace*{\fill}
{\vrule width6pt height6ptdepth0pt}\par}

\allowdisplaybreaks[4]

\begin{document}

\title{ On Bernstein Type Exponential Inequalities for Matrix Martingales}
 \date{\today}

\author{Zijie Tian\footnote{School of Mathematics, Shandong University, Jinan, 250100, China.}}

\date{\today}
\maketitle

\begin{abstract}
  In this work, Bernstein's concentration inequalities for squared integrable matrix-valued discrete-time martingales are obtained. Based on Lieb's theory and Bernstein's condition, a suitable supermartingale can be constructed. Our proof is largely based on this new exponential supermartingale, Freedman's method, and Doob's stopping theorem. Our result can be regarded as an extension of Tropp's work (ECP, 2012). 

 \end{abstract}

\section{Introduction}

There have been a lot of research achievements around concentration inequalities. The reader is referred to an excellent book: \textit{Concentration Inequalities for Sums and Martingales} [\cite{Ci}] which gives a detailed introduction of concentration inequalities. The concentration inequality is a class of inequality in probability, which describes the concentration phenomena of the values of the random variable. With a faster convergence rate, exponential type inequalities are of great importance when investigating the law of large numbers and the law of iterated logarithm.

As is known to all, a lot of exponential types inequalities are well known and frequently employed in statistics and probability. Especially, when considering the partial sum of independent random variables, there are multiple classical inequalities such as Petrov, Hoeffding [\cite{Ho}], Bennett [\cite{Bn1}, \cite{Bn2}] and Bernstein [\cite{Bs1}]. Bernstein's inequality is crucial because it gives an exponential upper bound on the tail probability of a large class of random variables.

Let us start by a traditional Bernstein's inequality. Suppose $(\Omega,\mathscr{F}
,P)$ is the probability space which is so large that we can construct all random objects of interest in it, and $X_1,\cdots,X_n$ be a finite sequence of centered independent variables with finite variances. Then let $S_n=X_1+X_2+\cdots +X_n$ and $\nu^2 =\dis \frac{1}{n} Var\bigg(\sum_{k=1}^{n}X_k\bigg)$, if there exists a constant $c>0$ such that for all $p>2$,

\begin{equation}\label{eq1}
	\mathbb{E}\big[|X_k|^p\big]\le \frac{p!c^{p-2}}{2} \mathbb{E} \big[X_k^2\big]\,,
\end{equation}
then

\begin{equation}\label{eq2}
	P\big(S_n \ge \sqrt{n}x\big)\le \exp\bigg(-\frac{\sqrt{n}x^2}{2(\sqrt{n}\nu^2+cx)}\bigg)\,,\,\forall \,x>0.	
\end{equation}

It is worth pointing out that (\ref{eq1}) is called Bernstein condition and the estimation of each moment of the random variable sequence is required. In fact, this condition can be reduced or even to the case of bounded random variables.

When considering random matrices, there are some corresponding results. The reader is referred to an article [\cite{TUF}], which gives an elaborate introduction of random matrices. There are some results such as Matrix Bennett, Bernstein, Hoeffding, Azuma, and McDiarmid [\cite{TUF}, \cite{TM}], and the Bernstein is also what we are concerned about.

Here and subsequently, we will introduce the semidefinite partial order $\bx$, which means
\[
\cu A \bx \cu B \quad \mathrm{if\,and \,only\, if}\quad \cu B - \cu A \,\mathrm{\, \,is \,positive\, semidefinite}. 
\]
The notation $\cu O \bx \cu A$ means that $\cu A$  is positive semidefinite. It is easy to check that $\lambda_{\mathrm{max}} \cu A \le \mathrm{tr} \cu A$ if $\cu A$ is positive semidefinite.

The expectation and the conditional expectation of a random matrix $\cu X$ are defined as follows. If $\cu X=(\xi_{ij})_{n\times n}$, and let $(\Omega,\mathscr{F},\{\mathscr{F}_n\}_{n\ge 0},P)$ be a probability space with the flow, then

\[
\mathbb{E}[\cu X]=\big(\mathbb{E} \xi_{ij}\big)_{n \times n}\quad,\quad 
\mathbb{E}\big[\cu X |\mathscr{F}_{n}\big]=\bigg(\mathbb{E}  \big[  \xi_{ij}|\mathscr{F}_n\big]\bigg)_{n \times n}\,.
\] 

Consider $\{\cu{X}_n \}_{n\ge 0}$ are independent, random, self-adjoint matrices with dimension $d$, which satifies
\begin{equation}
	E[\cu{X}_k]=\cu O \,,\, \lambda_{\mathrm{max}}(\cu X_k)\le C\,,\,a.s.
\end{equation}
Let $\sigma^2= \dis\Arrowvert \dis\sum_{k=1}^{n}\ \mathbb{E}(\cu X_k^2) \dis\Arrowvert$, where $\Arrowvert \cdot \Arrowvert$ is the spectral norm, which means the largest singular value of a matrix.
Then for all $t\ge 0$:
\begin{equation}
	P\bigg(\lam \big(\sum_{k=1}^{n} \cu{X}_k\big) \ge t\bigg) \le d\cdot \mathrm{exp}\bigg(-\frac{t^2}{2(\sigma^2 +Rt/3)}\bigg).
\end{equation}

Another perspective is the inequalities of martingales such as Azuma-Hoeffding [\cite{Ci},\cite{Ho}] , Freedman [\cite{Fre}] , Bernstein [\cite{Bs2}] and de la Pe\~na's [\cite{De}] inequalities. We emphasize that Bernstein's inequality of martingales gives an estimation of the tailed bound of the square-integrable martingale which satisfied the Bernstein condition.  Suppose $ M$ is an adapted square-integrable process, which satisfies $ {M_0}= 0$, then
\begin{equation*}
	\widetilde{  {M_n} } = \sum_{k=1}^{n} \mathbb E\bigg[\,{( {M_k}- {M_{k-1}})}^2\,|\,\mathscr F_{k-1}\,\bigg],
\end{equation*}

\begin{equation*}
	{V_n}= \widetilde{{M_n}}  - \widetilde{ {M_{n-1}} } = \mathbb E\bigg[\,{({M_n}-{M_{n-1}})}^2\,|\,\mathscr F_{n-1}\,\bigg].
\end{equation*}

Let $(M_n)_{n \ge 0}$ be a square-integrable martingale such that $M_0=0$. Assume that there exists a positive constant $c$ such that, for any integer $p \ge 3$ and all $1\le k\le n$,
\begin{equation}
	\qw \bigg[{(\Delta {M_n})^p}| \mathscr{F}_{k-1}\bigg] \le \frac{p!c^{p-2}}{2}  V_k,
\end{equation}
Then, for any positive $x$ and any positive $y$,
\begin{equation}
	\begin{aligned}
		P(M_n \ge nx,\widetilde{M}_n \le ny) &\le \,{\bigg(1+\frac{x^2}{2(y+cx)}\bigg)}^n \mathrm{exp}\bigg(-\frac{nx^2}{y+cx}\bigg)\\
		&\le  \,\mathrm{exp} \bigg(-\frac{nx^2}{2(y+cx)}\bigg).
	\end{aligned}
\end{equation} 

The main result of the passage is to generalize the traditional inequalities to matrix-valued processes. One of the most important matrix-valued processes is the matrix martingale. Let $(\Omega,\mathscr{F},\{\mathscr{F}_n\}_{n\ge 0},P)$ be a probability space with the flow. Suppose 

\[ \mathbf{X}_n = \left(
\begin{array}{cccc}
	\xi^{(n)}_{11} & \xi^{(n)}_{12} & \ldots & \xi^{(n)}_{1d}\\
	\xi^{(n)}_{21} & \xi^{(n)}_{22} & \ldots & \xi^{(n)}_{2d}\\
	\vdots & \vdots & \ddots & \vdots\\
	\xi^{(n)}_{d1} & \xi^{(n)}_{d2} & \ldots & \xi^{(n)}_{dd}\\
\end{array} \right) \]

is a $d$-dimensional random matrix, a matrix martingale is an adapted matrix-valued stochastic process $\{\cu X_n\}$ which satisfies $\forall 1\le i,j \le d$,
\begin{equation}
	E\big[\xi^{(n)}_{ij}|\mathscr{F}_{n-1}\big]=\xi^{(n-1)}_{ij}\,,\, E\big[\xi^{(n)}_{ij}] < \infty.
\end{equation}
That is, matrix martingale refers to a family of random matrices whose matrix elements are martingales.

It is of interest to know whether there are some corresponding inequalities in the adapted matrix-valued process or matrix martingale case. However, there are not many attempts has been made here to develop the traditional inequalities of matrix-valued process or matrix martingales. The Azuma inequality [\cite{TUF}] has its matrix-valued and matrix martingale version, and Oliveira  [\cite{Oli}]  has established an analog of Freedman’s inequality in the matrix setting who showed that the tail bound of maximum eigenvalue of the martingale is similar to Freedman’s inequality, and Tropp established a sharper version [\cite{Joel}]. We wish to investigate Bernstein's inequality in matrix-valued process or matrix martingales, and this paper aims to extend the results of Tropp [\cite{Joel}] to Bernstein's inequality of squared integrable matrix martingales. 

Our main tool is Lieb's theorem [\cite{TM}] which can help us complete one crucial step of the proof. To proof our results, we begin by establishing a supermartingale by Lieb's theorem [\cite{Lieb}] which has a suitable lower bound. We next define a stopping time and the basic idea is to apply Doob's stopping time theorem to estimate the eigenvalue of the maximum of squared integrable matrix martingales.

\section{Main Results}

To illustrate our notation, we will introduce the matrix exponential and the matrix logarithm. For Hermitian matrix $\cu A$, we can introduce the matrix exponential $e^{\cu A}$ by defining 
\begin{equation*}
	e^{\cu A} = \cu I + \sum_{n=1}^{\infty} \frac{{\cu A}^n}{n!}\,,
\end{equation*} 
and the matrix logarithm by defining the functional inverse of the matrix exponential: $\log (e^{\cu A})=\cu A$. There is no loss of generality in assuming that all the matrices in logarithmic functions are positive definite.

Suppose that $\cu{M}_n$ is a real symmetric random matrix with dimension $d$, and the process $\{ \cu{M}_n\}$ is squared integrable. We next denote $	\Delta \cu {M_n} = \cu {M_n}-\cu {M_{n-1}}\,,$
\begin{equation*}
	\widetilde{ \cu {M_n} } = \sum_{k=1}^{n} \mathbb E\bigg[\,{(\cu {M_k}-\cu {M_{k-1}})}^2\,|\,\mathscr F_{k-1}\,\bigg]\,,
\end{equation*}
\begin{equation*}
	\cu{V_n}= \widetilde{\cu {M_n}}  - \widetilde{ \cu {M_{n-1}} } = \mathbb E\bigg[\,{(\cu {M_n}-\cu {M_{n-1}})}^2\,|\,\mathscr F_{n-1}\,\bigg].
\end{equation*}

Throughout the passage, $\cu \varLambda_{\cu {X}}(t)$ stands for the binary function with matrix $\cu X$ and $t$:
\begin{equation}
	\cu \varLambda_{\cu X}(t)=\text{log}\bigg(\cu I + \frac{t^2 \cu X }{2(1-tc)}\bigg)\,,
\end{equation}
where $t>0$,$0<ct<1$ and $\cu I$ means the unit matrix.  Our main result reads as follows.

\begin{theorem}[Bernstein's Inequality for Matrix Martingales]\label{thm1}
	
	Let $\cu {M}$ be a squared integrable matrix martingale, $\cu {M_0}=\cu O$\,,\,and $\cu {M_n}$ is a real symmetric random matrix with dimension $d$ for all $n$, which satifies
	\\$\forall t>0 \,\, and \,\, 0<ct<1,$
	\begin{equation}
		\mathbb E \bigg[{(\Delta \cu {M_n})^p} \,|\,\mathscr F_{n-1}\bigg] \preccurlyeq \frac{p!c^{p-2}}{2} \cu V_n ,
	\end{equation}
	where $\cu{V_n}=\mathbb E\big[\,{(\cu {M_n}-\cu {M_{n-1}})}^2\,|\,\mathscr F_{n-1}\,\big]$,
	then forall $x,y,t>0\,,\,0<ct<1$,
	\begin{equation*}
		\mathbb P \bigg\{ \exists n :\, \lambda_{\mathrm{max}} (\cu {M_n}
		)\ge nx\,,\,\lambda_{\mathrm{max}}\bigg(\sum_{k=1}^{n} \cu \varLambda_{\cu {V_k}}(t)\bigg) \le n\mathrm{log}\bigg(1+\frac{yt^2}{2(1-tc)}\bigg) \bigg\} 
	\end{equation*}
	\begin{equation}
		\le d\,{\bigg(1+\frac{x^2}{2(y+cx)}\bigg)}^n \mathrm{exp}\bigg(-\frac{nx^2}{y+cx}\bigg)	\le d \,\mathrm{exp} \bigg(-\frac{nx^2}{2(y+cx)}\bigg).
	\end{equation}
	
\end{theorem}

\section{Proof of Main Results}

Here are some lemmas and tools for our demonstration.

\begin{lemma}{[\cite{DP}]}\label{lm3}
	For any Hermitian matrices $\cu A$ and $\cu B$ with the same dimension, if $A\bx B$,then\begin{equation}
		\mathrm{tr}e^{\cu A} \le \mathrm{tr}e^{\cu B}.
	\end{equation}
\end{lemma}

This theorem is an important result of Elliott Lieb on the convexity properties of
the trace exponential function. See [\cite{DP}] for a short proof of the fact.
\begin{lemma}{[\cite{Lieb}]}\label{thm4}
	For any fixed Hermitian matrix $\cu B$ with dimensional $d$, The function $f$:\begin{equation}
		f(\cu A)=\mathrm{trexp}\,(\cu B+\mathrm{log} \cu A)
	\end{equation}
	is a concave function on the convex cone of $d\times d$ positive-definite matrices.
\end{lemma}

This lemma is from [\cite{Lieb}]. And the readers can see [\cite{TUF}] and [\cite{JAF}] for additional discussion for this lemma. From the theorem mentioned above and Jensen's Inequality, the following lemma is obtained.
\begin{lemma}[Lieb]{[\cite{UF}]}\label{thm5}
	For any fixed Hermitian matrix $\cu B$ with dimensional $d$, if $\cu X$ is a random Hermitian matrix of the same dimension, we have
	\begin{equation}
		\mathbb{E} \bigg[\tre(\cu B+\cu X) \bigg] \le \tre \bigg(\cu B+\mathrm{log} \mathbb{E} e^{\cu X}\bigg).
	\end{equation}
\end{lemma}
For the proofs, we refer the reader to [\cite{UF}].

\begin{lemma}[Logarithm is Operator Monotone]{[\cite{TM}]}
	Suppose $\cu A$ and $\cu B$ are positive-definite matrices. If $\cu A\preccurlyeq \cu B $, then $\mathrm{log\mathbf{A}} \preccurlyeq \mathrm{log\mathbf{B}}$.
\end{lemma}

See [\cite{TM}] for a short proof of this lemma.

Now we will give the proof of our main theorem. To prove our inequality, we need to construct a suitable supermartingale, and it is crucial to estimate its lower bound. In this section, we follow the notation mentioned above. 

Define a new stochastic process $\{S_n\}_{n \ge 0}$ :
\begin{equation}
	S_n=G_t\,(\cu{M_n}\,,\,\cu{V_n}) = \mathrm{tr\,exp}\bigg[t\cu {M_n}-\sum_{k=1}^{n} \cu \varLambda_{\cu {V_k}}(t)\bigg]\,.
\end{equation}
We will consider the behavior of the $S_n$ defined above. The task is now to find whether $\{S_n\}_{n \ge 0}$ is a supermartingale. In fact, we have the following theorem.
\begin{theorem}
	The stochastic process $\{S_n\}_{n \ge 0}$\,,\,$S_0=d$ is a supermartingale with the Bernstein's condition 
	\begin{equation}
		\mathbb E \bigg[{(\Delta \cu {M_n})^p} \,|\,\mathscr F_{n-1}\bigg] \preccurlyeq \frac{p!c^{p-2}}{2} \cu V_n ,
	\end{equation}
	where the $d$ is the dimension of $\cu M_n$.
\end{theorem}
\begin{proof}
	TTo prove that the process $\{S_n\}_{n \ge 0}$ is a supermartingale, we need to prove that $\mathbb{E}[S_n\,|\,\mathscr F_{n-1}] \le S_{n-1}$. Because $\cu{V_n}=\mathbb E\big[\,{(\cu {M_n}-\cu {M_{n-1}})}^2\,|\,\mathscr F_{n-1}\,\big]$ is measurable to $\mathscr F_{n-1}$, we can still use Lemma 3 [\cite{Joel}], which means
	\begin{equation}
		\begin{aligned}
			\mathbb{E}[S_n\,|\,\mathscr F_{n-1}] &= \mathbb{E}\bigg[\,\mathrm{tr\,exp}\,\big[t\cu {M_n}-\sum_{k=1}^{n} \cu \varLambda_{\cu {V_k}}(t)\big]\,\bigg|\,\mathscr F_{n-1}\bigg]\\
			&\le \mathrm{tr\,exp} \bigg[ \mathrm{log\mathbb{E}}\big[ \mathrm{exp}\{t\cu {M_n}\}\,\big|\,\mathscr F_{n-1}\big]-\sum_{k=1}^{n} \cu \varLambda_{\cu {V_k}}(t) \bigg].
		\end{aligned}
	\end{equation}
	From the properties of conditional expectation, we have
	\[
	\mathrm{log\mathbb{E}}\big[ \mathrm{exp}\{t\cu {M_n}\}\,\big|\,\mathscr F_{n-1}\big]=\mathrm{log} \bigg[t\cu {M_{n-1}}\cdot \mathbb{E} \big[ \mathrm{exp}\{t\cu { \Delta M_n}\}\,\big|\,\mathscr F_{n-1}\big] \bigg].
	\]
	From Lemma1 and a brief calculation, we only need to show that
	\begin{equation}\label{log}
		\mathrm{log} \mathbb{E} \big[ \mathrm{exp}\{t\cu { \Delta M_n}\}\,\big|\,\mathscr F_{n-1}\big] \preccurlyeq \cu \varLambda_{\cu V_n}(t)=\text{log}\bigg(\cu I + \frac{t^2 \cu V_n }{2(1-tc)}\bigg),
	\end{equation}
	where $\cu{V_n}=\mathbb E\big[\,{(\cu {M_n}-\cu {M_{n-1}})}^2\,|\,\mathscr F_{n-1}\,\big]$.
	By definition of the matrix exponential, we have 
	\begin{equation}
		\mathbb{E} \big[ \mathrm{exp}\{t\cu { \Delta M_n}\}\,\big|\,\mathscr F_{n-1}\big]= \cu I +\sum_{p=2}^{\infty}\frac{t^{p}}{p!} \mathbb E \bigg[{(\Delta \cu {M_n})^p} \big|\,\mathscr F_{n-1} \bigg]\,,
	\end{equation}
	here $\mathbb E\big[ \Delta \cu {M_n} \big|\,\mathscr F_{n-1} \big]=0$ because $\cu M_n$ is a martingale. Then from the Bernstein's condition, we have
	
	\begin{equation}
		\mathbb{E} \big[ \mathrm{exp}\{t\cu { \Delta M_n}\}\,\big|\,\mathscr F_{n-1}\big] \preccurlyeq \cu I + \frac{\cu V_n t^2}{2(1-tc)}.
	\end{equation}
	
	Then from Lemma 4, (\ref{log}) has been proved, which means 
	
	\begin{equation}\begin{aligned}
			\mathbb{E}[S_n\,|\,\mathscr F_{n-1}] &= \mathbb{E}\bigg[\,\mathrm{tr\,exp}\,\big[t\cu {M_n}-\sum_{k=1}^{n} \cu \varLambda_{\cu {V_k}}(t)\big]\,\bigg|\,\mathscr F_{n-1}\bigg] \\
			&\le \mathrm{tr\,exp}\bigg[t\cu {M_{n-1}}-\sum_{k=1}^{n-1} \cu \varLambda_{\cu {V_k}}(t)\bigg]\\
			&= S_{n-1}.
	\end{aligned}\end{equation}
	
	Hense $\{S_n\}_{n \ge 0}$ is a supermartingale, and it is easily obtained that $S_0=\mathrm{tr\,exp}\,\cu O = \mathrm{tr} \, \cu I = d$.
\end{proof}

Next, we will give an estimation of the lower bound of $S_n$. Our next claim as follows. 

\begin{theorem}
	If $\forall x,y,t>0\,,\,0<ct<1$,
	\begin{equation}
		\lambda_{\mathrm{max}} (\cu {M_n}
		)\ge nx\,,\,\lambda_{\mathrm{max}}\bigg(\sum_{k=1}^{n} \cu \varLambda_{\cu {V_k}}(t)\bigg) \le n\mathrm{log}\bigg(1+\frac{yt^2}{2(1-tc)}\bigg)\,,
	\end{equation}
	then
	\begin{equation}
		S_n=\mathrm{tr\,exp}\bigg[t\cu {M_n}-\sum_{k=1}^{n} \cu \varLambda_{\cu {V_k}}(t)\bigg] \ge \mathrm{exp}\bigg[n(tx-\Lambda_y(t) )\bigg]\,.
	\end{equation}
\end{theorem}
\begin{proof}
	FFrom the conditions above, we have 
	\begin{equation}
		\begin{aligned}
			S_n&=\mathrm{tr\,exp}\,\bigg[t\cu {M_n}-\sum_{k=1}^{n} \cu \varLambda_{\cu {V_k}}(t)\bigg]\\
			&\ge \mathrm{tr\,exp} \bigg[t\cu {M_n}-n\mathrm{log}(1+\frac{yt^2}{2(1-tc)}) \cu I\bigg] \\
			& \ge \lambda_{\mathrm{max}} \, \mathrm{exp} \bigg[t\cu {M_n}-n\mathrm{log}(1+\frac{yt^2}{2(1-tc)}) \cu I\bigg]\\ 
			& \ge \mathrm{exp}\bigg[tnx-n\mathrm{log}(1+\frac{yt^2}{2(1-tc)})\bigg]\\
			& = \mathrm{exp}\bigg[n(tx-\Lambda_y(t) )\bigg]\,.
		\end{aligned}
	\end{equation}
	
	The first inequality depends on
	\begin{equation}
		\sum_{k=1}^{n} \cu \varLambda_{\cu {V_k}}(t) \bx \lam\bigg(\sum_{k=1}^{n} \cu \varLambda_{\cu {V_k}}(t)\bigg) \cu I
	\end{equation}
	and Lemma 1. The second inequality depends on the fact that $e^{\cu X}$ is semidefinite, and the trace of a positive definite matrix is greater than the maximum eigenvalue. The third inequality is based on the spectral mapping theorem and some properties of the maximum eigenvalue map.
	
\end{proof}

We next turn to prove our main result by Doob's stopping theorem.
Let us denote by $A_n$ the set
\begin{equation}
	A_n=\bigg\{ \lambda_{\mathrm{max}} (\cu {M_n}
	)\ge nx\,,\,\lambda_{\mathrm{max}}\bigg(\sum_{k=1}^{n} \cu \varLambda_{\cu {V_k}}(t)\bigg) \le n\mathrm{log}\big(1+\frac{yt^2}{2(1-tc)}\big) \bigg\} \,.
\end{equation}
Then $\displaystyle A=\bigcup_{n=1}^{\infty} A_n$. In order to get our inequality, it is necessary to introduce a stopping-time
\begin{equation}
	\tau=\mathrm{inf}\bigg\{  n \ge 0\,:\, \lambda_{\mathrm{max}} (\cu {M_n}
	)\ge nx\,,\,\lambda_{\mathrm{max}}\bigg(\sum_{k=1}^{n} \cu \varLambda_{\cu {V_k}}(t)\bigg) \le n\mathrm{log}(1+\frac{yt^2}{2(1-tc)}) \bigg\}. 
\end{equation}

From the Doob's stopping theorem the $S_{\tau \wedge n}$ is also a positive supermartingale with an initial value $d$. Using the fact that $ \tau< \infty$ on event $A$, by Fatou’s lemma we have 
\begin{equation}
	d\ge \varliminf_{n \to \infty} \mathrm{E}[S_{\tau \wedge n}] \ge \varliminf_{n \to \infty} \mathrm{E}[S_{\tau \wedge n}\,\mathbb{I}_A] \ge \mathrm{E}[\varliminf_{n \to \infty}\,S_{\tau \wedge n}\,\mathbb{I}_A]=\mathrm{E}[S_{\tau}\,\mathbb{I}_A].
\end{equation}
Then
\begin{equation}
	d\ge \mathrm{E}[S_{\tau}\,\mathbb{I}_A] \ge P(A)\cdot \inf_{A} S_{\tau}.
\end{equation}
From the lemma, by choosing $t=\displaystyle \frac{x}{y+cx}$ we finally have our inequality\begin{equation}
	P(A) \le d\cdot \mathrm{exp}\bigg[n(\Lambda_y(t)-tx)\bigg] \le d\,{\bigg(1+\frac{x^2}{2(y+cx)}\bigg)}^n \mathrm{exp}\bigg(-\frac{nx^2}{y+cx}\bigg).
\end{equation}
The main part of the theorem is proved.

\section*{Acknowledgments}

Thanks to Prof. Hanchao Wang for pivotal guidance and suggestions who helped me a lot with my academic, courses, and research projects. And thanks to the School of Mathematics, Shandong University that provided great support for the undergraduate research projects.

\bibliographystyle{amsplain}

\begin{thebibliography}{10}
	\bibitem{Ci}
	Bercu, B., Delyon, B. and Rio, E. \textit{Concentration inequalities for sums and martingales. SpringerBriefs in Mathematics}, Springer, (2015).
	
	\bibitem{Ho}
	Hoeffding, W. Probability inequalities for sums of bounded random variables. \textit{J. Amer. Statist. Assoc. }, \textbf{58}, 13-30, (1963).
	
	\bibitem{Bn1}
	Bennett, G. On the probability of large deviations from the expectation for sums of bounded independent random variables. \textit{Biometrika},\textbf{ 50}, 528-535, (1963).
	
	\bibitem{Bn2}
	Bennett, G. Probability inequalities for the sum of independent random variables. \textit{J. Amer. Statist. Assoc.},\textbf{ 57}, 33-45, (1962).
	
	\bibitem{Bs1}
	Bernstein, S.N. \textit{Theory of Probability}, Moscow. (1927).
	
	
	
	\bibitem{TUF}
	Joel A. Tropp. User-Friendly Tail Bounds for Sums of Random Matrices. \textit{Foundations of Computational Mathematics.},  \textbf{12(4)}:389-434,(2012).
	
	\bibitem{TM}
	Joel A. Tropp. An Introduction to Matrix Concentration Inequalities. arXiv:1501.01571v1 (2015).
	
	\bibitem{Fre}
	D. A. Freedman. On tail probabilities for martingales. \textit{Ann. Probab.}, \textbf{3(1)}:100–118, (1975).
	
	\bibitem{Bs2}
	Bernstein, S. Sur quelques modifications de l'in\'egalit\'e de Tchebycheff. C.R. (Doklady) Acad. Sci. URSS \textbf{17}, 279–282 (1937).
	
	\bibitem{De}
	De la Pe\~na, V. H.: A general class of exponential inequalities for martingales and ratios. Ann.  Probab. \textbf{27}, 537–564 (1999).
	
	\bibitem{Oli}
	R. I. Oliveira. Concentration of the adjacency matrix and of the Laplacian in random graphs with independent edges. Available at arXiv:0911.0600, (2010).
	
	\bibitem{Joel}
	Joel Tropp. Freedman's inequality for matrix martingales. \textit{Electron. Commun. Probab.}, \textbf{16}:262-270, (2011).
	
	\bibitem{Lieb}
	E. H. Lieb. Convex trace functions and the Wigner–Yanase–Dyson conjecture. \textit{Adv. Math.},\textbf{ 11}:267–288, (1973).
	
	\bibitem{DP}
	D. Petz, A survey of certain trace inequalities, in Functional Analysis and Operator Theory. Banach Center Publications, vol. 30 (Polish Acad. Sci., Warsaw, 1994), pp. 287–298.
	
	\bibitem{JAF}
	J. A. Tropp. From the joint convexity of quantum relative entropy to a concavity theorem of Lieb. Available at arXiv:1101.1070, (2010).
	
	\bibitem{UF}
	J. A. Tropp. User-friendly tail bounds for matrix martingales. ACM Report 2011-01, California Inst. Tech., Pasadena, CA, (2011).
	
	%
	%
	%
	%
	%
	%
	%
	%
	%
	%
	%
	%
	%
	%
	%
	%
	%
	%
	%
	%
	%
	%
	%
	%
	%
	%
	%
	%
	%
	%
	%
	%
	%
	%
	%
	%
	%
	%
	%
	%
	%
	%
	%
	%
	%
	%
	%
	%
	%
	
	
	
	
	
	
\end{thebibliography}

\end{document}